\documentclass[11pt]{article}
\usepackage[margin=1in]{geometry}
\usepackage{amsmath,amssymb,amsthm,mathtools,graphicx,booktabs,cite,multirow,array}
\usepackage{microtype}
\usepackage{xspace}
\usepackage[hidelinks]{hyperref}
\usepackage{enumitem}
\usepackage{algorithm}
\usepackage{algpseudocode}
\usepackage{placeins}
\usepackage{indentfirst}
\newtheorem{theorem}{Theorem}
\newtheorem{corollary}{Corollary}
\newtheorem{lemma}{Lemma}
\newcommand{\FP}{\operatorname{FP}}
\newcommand{\cE}{\mathcal E}
\newcommand{\cB}{\mathcal B}
\newcommand{\one}{\boldsymbol 1}
\newcommand{\alg}{\textsc{SR-CG}\xspace}

\title{Spectral-Residual Continuous Greedy for Tensor Sampling}
\author{Hao Li, Jie Xu, Zheng Xie\thanks{Corresponding author: xiezheng81@nudt.edu.cn.}\\
College of Science, National University of Defense Technology, Changsha 410073, China}

\date{September 2026}

\begin{document}
\maketitle

\begin{abstract}
Sampling a multidomain tensor from limited measurements is fundamental
in structured linear inverse problems. Kronecker-structured sampling
avoids the full high-dimensional sensing matrix, but design remains
difficult: sequential discrete methods can commit the cross-mode budget
too early, whereas standard continuous greedy avoids such early
commitment at the cost of repeated state-dependent gradient evaluations.
We propose spectral-residual continuous greedy (\alg) for frame-potential
(FP) tensor sampling. \alg maintains a state-dependent fractional
allocation before rounding. Mode-wise Gram matrices provide safe
gradient intervals for direction certification, while Shapley values
prioritize unresolved gradient queries.
SR-CG then applies deterministic rounding followed by path-guided
exchange (PGX), which reuses the final fractional state to restrict
candidate swaps and accepts only exact FP-decreasing exchanges.
We establish a finite-step approximation guarantee
that approaches the classical $1-1/e$ factor as direction certification
becomes exact and finite-step residual error vanishes. Experiments show
fewer exact gradient evaluations and better FP designs, with clearer
gains on instances where the cross-mode budget allocation is difficult
to determine. \alg also achieves lower average normalized mean-squared error (NMSE)
than Greedy-FP at all tested noise levels, although the
reconstruction gain is smaller than the FP gain.
\end{abstract}


\section{Introduction}
Choosing a small set of informative measurements is central to sensor
placement, experimental design, graph sampling, and structured inverse
problems \cite{joshi2009sensor,shamaiah2010greedy,krause2008sensor}.
For tensor data, selections made in different modes jointly determine
the observed entries and the geometry of the sensing system. A
Kronecker-structured sampler exploits a known multilinear model to
choose indices mode by mode without constructing the full
high-dimensional sensing matrix \cite{ortiz2019sparse}. The total modal
budget is fixed, but the individual modal cardinalities are not. The
design must jointly determine both the cross-mode allocation and the
retained rows within each mode. Frame potential (FP) is attractive
in this setting because it avoids repeated matrix inversion and
factorizes over Kronecker-structured sensing matrices
\cite{ranieri2014near,benedetto2003finite,ortiz2019sparse}.

Greedy-FP \cite{ortiz2019sparse} removes one feasible row at a time
using the current FP marginal over the shared cross-mode feasible set.
Its final modal cardinalities arise from sequential greedy removals
rather than being fixed beforehand, although this discrete scoring can
commit the budget to one mode early. Fast Frank--Wolfe (FFW) \cite{li2026fast} also leaves
the modal cardinalities free. It evaluates normalized mode-wise
multilinear-FP derivatives at the midpoint, compares the scores across
modes, and selects under the common total budget and per-mode minimum
counts. This cross-mode ranking is computed once and is not updated as
the retained design changes. A fast sampling method based on mean-squared error (MSE) provides a related
alternative for linear inverse models \cite{wang2022fast}. Continuous
greedy offers a global fractional search with the classical $1-1/e$
guarantee under a matroid constraint
\cite{calinescu2011maximizing}, but a direct tensor implementation
repeatedly evaluates state-dependent gradient coordinates over many
short steps.

We develop spectral-residual continuous greedy (\alg) to separate
state-dependent cross-mode allocation from expensive exact-gradient
evaluation. Unlike Greedy-FP's sequential discrete commitment and FFW's
fixed midpoint ranking, \alg keeps the cross-mode allocation fractional
and state dependent until rounding. Its spectral intervals are used to
certify the current linear-oracle direction without evaluating every
coordinate exactly, while Shapley values only order unresolved gradient
queries and do not act as sampling scores. Hence the accepted direction
remains determined by the current state-dependent gradient rather than
by a fixed global ranking. The resulting finite-step analysis provides
an approximation guarantee that approaches the classical $1-1/e$ factor as
direction certification becomes exact and finite-step residual error vanishes.

Our main contributions are:
\begin{enumerate}[leftmargin=1.8em]
\item We derive a mode-wise spectral representation of the tensor-FP gradient and safe coordinate intervals that use only $K_r\times K_r$ matrices.
\item We turn these intervals into an online certificate for the quality of the current matroid direction, so exact gradient coordinates are evaluated only when needed.
\item We derive an exact directional FP polynomial and use it to choose variable steps with an explicit residual condition, leading to a finite-step approximation guarantee that approaches $1-1/e$ as direction certification becomes exact and finite-step residual error vanishes.
\item
We complete SR-CG with guarantee-preserving deterministic rounding and
a problem-specific path-guided exchange (PGX) stage. PGX reuses the
final fractional state to restrict exact FP-decreasing exchanges,
coupling the fractional and discrete stages without weakening the
approximation guarantee.
\end{enumerate}

\subsection{Related work}
Convex and submodular sensor selection is classical \cite{joshi2009sensor,shamaiah2010greedy,krause2008sensor}. FP-based selection avoids repeated inversion and has been used for sensor placement and sampling \cite{ranieri2014near,benedetto2003finite}. For Kronecker-structured tensor sampling, Ortiz-Jimenez et al. \cite{ortiz2019sparse} showed that tensor FP factorizes across modes and proposed Greedy-FP, which repeatedly updates the feasible FP marginal over all modes and lets the modal cardinalities emerge from sequential discrete removals. Wang et al. \cite{wang2022fast} developed a fast sampling method based on MSE for linear models. Li et al. \cite{li2026fast} proposed FFW: for tensor sampling it computes a midpoint multilinear-FP derivative in each mode, applies the original FFW mode-wise normalization, and selects the required rows from the resulting cross-mode scores subject to the per-mode minimum counts. Thus FFW also determines the cross-mode allocation. Its distinction is that the normalized ranking is computed once rather than updated with the retained design. We use these methods as baselines and keep their roles separate from the proposed adaptive continuous search.

Exchange-based refinement is well established in experimental design.
Lau and Zhou analyze local-search and randomized exchange frameworks,
including the classical Fedorov exchange method
\cite{lau2022local}. In sensor selection, Uci{\'n}ski combines a relaxed
design with randomization and a restricted exchange algorithm
\cite{ucinski2020doptimal}.
These works establish the general exchange and relaxation-guided
refinement paradigm. Within SR-CG, PGX provides a problem-specific
coupling between the fractional and discrete stages: the final
fractional state restricts the removal frontier, $\phi_e$ provides only
a deterministic secondary ordering, and feasible swaps are ultimately
evaluated by the exact Kronecker-factorized FP.

For monotone submodular maximization under matroid constraints, continuous greedy gives the classical $1-1/e$ factor \cite{calinescu2011maximizing}, and pipage or swap rounding converts a fractional solution to a feasible discrete set \cite{chekuri2010dependent}. Recent work develops faster algorithms that reduce query, communication, or gradient cost in more general settings \cite{buchbinder2024deterministic,rostami2026atcg}. Complementary lower-bound results show that near-linear query cost is unavoidable for broad submodular maximization problems \cite{peng2025lower}. Our method does not claim a general sublinear oracle bound. Instead, it uses the special algebra of tensor FP to certify many coordinate comparisons from small mode-wise matrices. Shapley values have also been used for sensor importance \cite{ravindra2025sensor}. Here Owen's integral representation gives a fixed analytic priority used for query ordering and as a secondary PGX ordering. It does not determine the final sampling rule.

\section{Problem formulation}
We define the tensor model, structured sampler, FP objective, and equivalent
submodular complement problem. Lowercase bold italic letters denote vectors,
uppercase bold upright letters denote matrices, calligraphic letters denote
sets, and plain italic letters denote scalars or scalar-valued functions.

\subsection{Tensor model and structured sampling}
Let an order-$R$ signal tensor $\mathcal F\in\mathbb R^{N_1\times\cdots\times N_R}$ admit the known multilinear model
\begin{equation*}
 \mathcal F=\mathcal G\times_1\mathbf U_1\times_2\cdots\times_R\mathbf U_R,
\end{equation*}
where the factor matrices $\mathbf U_r\in\mathbb R^{N_r\times K_r}$ are known and have full column rank, while the compact core $\mathcal G\in\mathbb R^{K_1\times\cdots\times K_R}$ is the unknown low-dimensional representation to be reconstructed. We write the $i$th row of $\mathbf U_r$ as $(\boldsymbol u_i^{(r)})^\top$. Thus $N_r$ is the number of candidate indices in mode $r$, and $K_r$ is the dimension of the corresponding latent factor. Vectorization gives
\begin{equation}
 \boldsymbol f=(\mathbf U_R\otimes\cdots\otimes\mathbf U_1)\boldsymbol g,
 \qquad \boldsymbol f=\operatorname{vec}(\mathcal F),\quad
 \boldsymbol g=\operatorname{vec}(\mathcal G).
 \label{eq:vector_model_full}
\end{equation}
This is the linear inverse model associated with the tensor.

For mode $r$, let $\mathcal N_r=[N_r]$ be the candidate row set and let $\mathcal L_r\subseteq\mathcal N_r$ be the retained rows. The row submatrix indexed by $\mathcal L_r$ is $\mathbf U_r(\mathcal L_r)$. The Cartesian product $\mathcal L_1\times\cdots\times\mathcal L_R$ identifies the tensor entries observed by a Kronecker-structured sampler, whose sensing matrix is
\begin{equation*}
 \boldsymbol{\Psi}(\mathcal L)=\mathbf U_R(\mathcal L_R)\otimes\cdots\otimes\mathbf U_1(\mathcal L_1),
 \qquad \mathcal L=(\mathcal L_1,\ldots,\mathcal L_R).
\end{equation*}
Hence $|\mathcal L_r|$ is the number of retained mode-$r$ indices, whereas $\prod_r|\mathcal L_r|$ is the number of tensor entries selected by the Cartesian-product sampler. With additive measurement noise, the selected linear system is used to estimate $\boldsymbol g$ by least squares when the sampled factors have full column rank, and \eqref{eq:vector_model_full} then reconstructs $\boldsymbol f$.

\subsection{Frame-potential sampling problem}
For any $\mathcal S\subseteq\mathcal N_r$, define the mode-$r$ frame potential by
\begin{equation*}
\begin{aligned}
 F_r(\mathcal S)
 &=\FP(\mathbf U_r(\mathcal S))
 =\operatorname{tr}\!\left[(\mathbf U_r(\mathcal S)^\top\mathbf U_r(\mathcal S))^2\right]\\
 &=\sum_{i,j\in\mathcal S}\left((\boldsymbol u_i^{(r)})^\top\boldsymbol u_j^{(r)}\right)^2.
\end{aligned}
\end{equation*}
The Kronecker identity makes the tensor FP factorize across modes \cite{ortiz2019sparse}:
\begin{equation*}
 F(\mathcal L)=\FP(\boldsymbol{\Psi}(\mathcal L))=\prod_{r=1}^R F_r(\mathcal L_r).
\end{equation*}
Let $N_\Sigma=\sum_rN_r$. We use the additive total modal budget employed in prior structured tensor-sampling work:
\begin{equation}
 \min_{\mathcal L_1,\ldots,\mathcal L_R}F(\mathcal L)
 \quad\text{s.t.}\quad
 \sum_{r=1}^R|\mathcal L_r|=L,
 \qquad |\mathcal L_r|\ge\kappa_r\ge K_r.
 \label{eq:budget}
\end{equation}
Here $L$ is the total number of retained \emph{modal indices}, and feasibility requires $\sum_r\kappa_r\le L\le N_\Sigma$. The total budget $L$ is fixed, whereas the individual modal cardinalities $|\mathcal L_r|$ are decision variables. Thus $\kappa_r$ is a lower bound, not a prescribed mode-wise budget. The optimization jointly chooses the allocation of $L$ across modes and the retained rows within each mode. Setting $\kappa_r>K_r$ adds sampling slack, while the standard rank-based constraint used in Greedy-FP and FFW is obtained with $\kappa_r=K_r$. Because $\mathbf U_r$ has $K_r$ columns, $\kappa_r\ge K_r$ is necessary for full-column-rank recovery but is not sufficient by itself. Therefore, reconstruction experiments check rank explicitly.

\subsection{Equivalent complement problem and submodular structure}
For the submodular form used later, let $\mathcal N=(\mathcal N_1,\ldots,\mathcal N_R)$ denote the full retained design and define the disjoint tagged ground set $\mathcal E=\bigcup_{r=1}^R(\{r\}\times\mathcal N_r)$. Set $B=N_\Sigma-L$. A tagged removal set $\mathcal C\subseteq\mathcal E$ has mode component $\mathcal C_r=\{i:(r,i)\in\mathcal C\}$ and retained design $\mathcal L(\mathcal C)=(\mathcal N_1\setminus\mathcal C_1,\ldots,\mathcal N_R\setminus\mathcal C_R)$. The feasible removal bases are
\begin{equation*}
 \cB_B=\left\{\mathcal C\subseteq\mathcal E:\ |\mathcal C|=B,\ |\mathcal C_r|\le N_r-\kappa_r\ \forall r\right\}.
\end{equation*}
Equivalently, the total removal budget $B$ is shared across modes. The upper bounds on $|\mathcal C_r|$ enforce only the retained lower quotas and do not prescribe a removal count for each mode.
With $Z_r=F_r(\mathcal N_r)$, define the normalized FP reduction for every $\mathcal C\subseteq\mathcal E$ by
\begin{equation}
 v(\mathcal C)=1-\frac{F(\mathcal L(\mathcal C))}{F(\mathcal N)}
 =1-\frac{\prod_{r=1}^RF_r(\mathcal N_r\setminus\mathcal C_r)}{\prod_{r=1}^RZ_r}.
 \label{eq:v}
\end{equation}
Minimizing \eqref{eq:budget} is equivalent to maximizing $v(\mathcal C)$ over $\cB_B$. This complement objective is normalized, monotone, and submodular under the truncated partition-matroid constraint \cite{ortiz2019sparse}. We use $v$ as the characteristic function of the cooperative game on the tagged rows.

The mode-wise retained FP is also monotone supermodular. For $i\notin\mathcal S$,
\begin{equation*}
 F_r(\mathcal S\cup\{i\})-F_r(\mathcal S)
 =\|\boldsymbol u_i^{(r)}\|_2^4
 +2\sum_{j\in\mathcal S}\left((\boldsymbol u_i^{(r)})^\top\boldsymbol u_j^{(r)}\right)^2,
\end{equation*}
which is nondecreasing as $\mathcal S$ grows. This identity records the known structural property used later.

\section{Spectral-residual continuous greedy}
The method combines an analytic Shapley reference, exact low-dimensional gradient formulas, safe spectral intervals, a certified linear oracle, residual-controlled adaptive steps, and a discrete completion stage. Shapley values are used only to prioritize unresolved computations. The accepted continuous direction is always determined by the current state-dependent gradient.

\subsection{Multilinear extension and exact modal gradient}
For each mode, define the row statistics
\begin{equation*}
 a_i^{(r)}=\|\boldsymbol u_i^{(r)}\|_2^4,\qquad
 b_i^{(r)}=\sum_{j\ne i}\left((\boldsymbol u_i^{(r)})^\top\boldsymbol u_j^{(r)}\right)^2.
\end{equation*}
These quantities depend only on the modal factor matrices and can be precomputed without an $N_r\times N_r$ pairwise table. Let $\mathbf G_r^{\rm full}=\mathbf U_r^\top\mathbf U_r$. Then
\[
 a_i^{(r)}+b_i^{(r)}
 =(\boldsymbol u_i^{(r)})^\top\mathbf G_r^{\rm full}\boldsymbol u_i^{(r)},
\]
so all $a_i^{(r)}$ and $b_i^{(r)}$ can be computed in $O(N_rK_r^2)$ time. Let $\boldsymbol x=(\boldsymbol x^{(1)},\ldots,\boldsymbol x^{(R)})\in[0,1]^{N_\Sigma}$ be a fractional removal vector and set $\boldsymbol y=\one-\boldsymbol x$. Thus $y_i^{(r)}$ is the $i$th entry of the mode-$r$ vector $\boldsymbol y^{(r)}$. Define the weighted modal Gram
\begin{equation*}
 \mathbf G_r(\boldsymbol x)=\mathbf U_r^{\top}\operatorname{diag}(\boldsymbol y^{(r)})\mathbf U_r.
\end{equation*}
The naive quantity $\|\mathbf G_r(\boldsymbol x)\|_F^2$ treats diagonal survival as $(y_i^{(r)})^2$, while a multilinear extension requires survival probability $y_i^{(r)}$. The exact correction is therefore
\begin{equation*}
 H_r(\boldsymbol x)=\|\mathbf G_r(\boldsymbol x)\|_F^2+
 \sum_i a_i^{(r)}(y_i^{(r)}-(y_i^{(r)})^2).
\end{equation*}
Equivalently,
\begin{equation}
 H_r(\boldsymbol x)=\sum_i a_i^{(r)}y_i^{(r)}+\sum_{i\ne j}\left((\boldsymbol u_i^{(r)})^\top\boldsymbol u_j^{(r)}\right)^2y_i^{(r)}y_j^{(r)}. \label{eq:Hexpanded}
\end{equation}
Hence
\begin{equation}
 \widetilde v(\boldsymbol x)=1-\prod_{r=1}^R\frac{H_r(\boldsymbol x)}{Z_r}. \label{eq:multiv}
\end{equation}

\begin{lemma}[Exact factorized gradient]\label{lem:grad}
For coordinate $(r,i)$,
\begin{equation}
 g_i^{(r)}(\boldsymbol x):=\partial_i^{(r)}\widetilde v(\boldsymbol x)
 =\gamma_r(\boldsymbol x)\left[a_i^{(r)}+2\sum_{j\ne i}\left((\boldsymbol u_i^{(r)})^\top\boldsymbol u_j^{(r)}\right)^2y_j^{(r)}\right], \label{eq:grad}
\end{equation}
where
\begin{equation*}
 \gamma_r(\boldsymbol x)=\frac1{Z_r}\prod_{s\ne r}\frac{H_s(\boldsymbol x)}{Z_s}.
\end{equation*}
\end{lemma}
\begin{proof}
Differentiate \eqref{eq:Hexpanded} with respect to $x_i^{(r)}$. Since $\partial y_i^{(r)}/\partial x_i^{(r)}=-1$,
\begin{equation*}
 -\partial_i^{(r)}H_r=a_i^{(r)}+2\sum_{j\ne i}\left((\boldsymbol u_i^{(r)})^\top\boldsymbol u_j^{(r)}\right)^2y_j^{(r)}. 
\end{equation*}
Applying the product rule to \eqref{eq:multiv} yields \eqref{eq:grad}.
\end{proof}

For a removal set $\mathcal C$, write $g(\mathcal C)=\sum_{(r,i)\in\mathcal C}g_i^{(r)}(\boldsymbol x)$ at the current state. No pairwise row table is needed to evaluate \eqref{eq:grad}. Indeed,
\begin{equation*}
 \sum_j\left((\boldsymbol u_i^{(r)})^\top\boldsymbol u_j^{(r)}\right)^2y_j^{(r)}=(\boldsymbol u_i^{(r)})^{\top}\mathbf G_r(\boldsymbol x)\boldsymbol u_i^{(r)},
\end{equation*}
so an exact queried coordinate is a $K_r\times K_r$ quadratic form.

\subsection{Shapley reference for the tensor FP-reduction game}
For the Owen-path calculation, set
\begin{equation*}
 A_r=\sum_i a_i^{(r)},\qquad D_r=\sum_i b_i^{(r)}.
\end{equation*}
Since $Z_r=F_r(\mathcal N_r)$, the definitions above give $Z_r=A_r+D_r$.

\subsubsection{Owen path representation}
Let $\widetilde v$ be the multilinear extension of $v$. Owen's identity gives
\begin{equation}
 \phi_e=\int_0^1 \partial_e\widetilde v(t\one)\,dt \label{eq:owen}
\end{equation}
for the Shapley value of element $e$. Along $\boldsymbol x=t\one$, write $u=1-t$. A diagonal row term survives with probability $u$, while an off-diagonal ordered pair survives with probability $u^2$. Therefore the expected retained FP of mode $s$ is
\begin{equation*}
 H_s(t\one)=A_su+D_su^2.
\end{equation*}
The marginal of removing row $(r,i)$ along the same path is
\begin{equation}
 \partial_i^{(r)}\widetilde v(t\one)=
 \frac{a_i^{(r)}+2b_i^{(r)}u}{Z_r}
 \prod_{s\ne r}\frac{A_su+D_su^2}{Z_s}. \label{eq:pathderiv}
\end{equation}
Define
\begin{equation}
 I_{rj}=\int_0^1u^j\prod_{s\ne r}\frac{A_su+D_su^2}{Z_s}\,du,
 \qquad j\in\{0,1\}. \label{eq:Irj}
\end{equation}

\begin{lemma}[Closed-form tensor Shapley value]\label{lem:phi}
The Shapley value of row $(r,i)$ for the normalized FP-reduction game \eqref{eq:v} is
\begin{equation}
 \phi_i^{(r)}=\frac{a_i^{(r)}I_{r0}+2b_i^{(r)}I_{r1}}{Z_r}. \label{eq:phi}
\end{equation}
\end{lemma}
\begin{proof}
Substitute \eqref{eq:pathderiv} into \eqref{eq:owen}, change variables from $t$ to $u=1-t$, and separate the coefficient of $a_i^{(r)}$ from that of $2b_i^{(r)}$. This gives exactly \eqref{eq:phi}. The product in \eqref{eq:Irj} has degree at most $2(R-1)$, so the integrals can be evaluated exactly from polynomial coefficients.
\end{proof}

For direct evaluation of the integrals in \eqref{eq:Irj}, expand the other-mode product as
\[
 \prod_{s\ne r}\left(\frac{A_s}{Z_s}u+\frac{D_s}{Z_s}u^2\right)
 =\sum_{d=0}^{2(R-1)}c_{r,d}u^d.
\]
It then follows that
\[
 I_{r0}=\sum_d\frac{c_{r,d}}{d+1},\qquad
 I_{r1}=\sum_d\frac{c_{r,d}}{d+2}.
\]
The implementation evaluates these expressions directly from the polynomial coefficients, so coalition enumeration is needed only for small-instance verification and not for the optimization algorithm.

\begin{corollary}[Vector FFW--Shapley identity]\label{cor:ffw}
As the $R=1$ specialization of Lemma~\ref{lem:phi},
\begin{equation}
 \phi_i^{(1)}=\frac{a_i^{(1)}+b_i^{(1)}}{Z_1}. \label{eq:vectorphi}
\end{equation}
\end{corollary}
\begin{proof}
For $R=1$, the product in \eqref{eq:Irj} is empty, hence $I_{10}=1$ and $I_{11}=1/2$. Substitution into \eqref{eq:phi} gives \eqref{eq:vectorphi}. Equation \eqref{eq:pathderiv} reduces to $(a_i^{(1)}+2b_i^{(1)}u)/Z_1$, whose integral on $[0,1]$ is its midpoint value.
\end{proof}
The normalized vector FFW score and the Shapley value therefore have the same numerical ordering. FFW keeps the rows with the smallest retained-design scores, whereas the complement game ranks removals by the corresponding largest benefits. The marginal path is affine in $u$, so its integral equals its midpoint value.

The tensor formula has a different interpretation. The coefficients $I_{r0}$ and $I_{r1}$ depend on every other mode. Thus rows in different modes are calibrated by the residual FP trajectories that coexist with them along random coalition-arrival times. Consequently, no independent mode normalization is needed to define $\phi$.

\subsection{Spectral gradient certificates}
The exact coordinate in Lemma~\ref{lem:grad} uses a $K_r\times K_r$ quadratic form. We first build a mean-state proxy, then tighten it with the current modal spectrum.

Let
\begin{equation}
 \bar y_r=\frac1{N_r}\sum_i y_i^{(r)},\qquad
 \Delta_r=\max_i|y_i^{(r)}-\bar y_r|,\qquad
 d_i^{(r)}=\|\boldsymbol u_i^{(r)}\|_2^2. \label{eq:ybar}
\end{equation}
The mean-path proxy and its valid radius are
\begin{equation*}
 q_{i,\mathrm{mean}}^{(r)}=\gamma_r(a_i^{(r)}+2b_i^{(r)}\bar y_r),\qquad
 \rho_{i,\mathrm{mean}}^{(r)}=2\gamma_rb_i^{(r)}\Delta_r.
\end{equation*}
This defines the mean-state certificate.

Using the full Gram defined above, introduce the centered state, scalar center, and centered residual
\[
 \mathbf M_r=\mathbf G_r(\boldsymbol x)-\bar y_r\mathbf G_r^{\rm full},\qquad
 \mu_r=\frac{\operatorname{tr}(\mathbf M_r)}{K_r},\qquad
 \mathbf R_r=\mathbf M_r-\mu_r\mathbf I_{K_r}.
\]
These quantities isolate the nonuniform component of the current fractional state from its mean-state approximation.

\begin{lemma}[Certified spectral gradient interval]\label{lem:spectral}
For every state $\boldsymbol x$ and coordinate $(r,i)$,
\begin{equation}
 g_i^{(r)}(\boldsymbol x)=\widehat q_i^{(r)}(\boldsymbol x)+2\gamma_r(\boldsymbol x)(\boldsymbol u_i^{(r)})^\top \mathbf R_r(\boldsymbol x)\boldsymbol u_i^{(r)}, \label{eq:specdecomp}
\end{equation}
where
\begin{equation}
 \widehat q_i^{(r)}=\gamma_r\!\left[a_i^{(r)}+2b_i^{(r)}\bar y_r+2\mu_rd_i^{(r)}-2a_i^{(r)}(y_i^{(r)}-\bar y_r)\right]. \label{eq:qhat}
\end{equation}
Therefore,
\begin{equation}
 |g_i^{(r)}-\widehat q_i^{(r)}|\le \rho_{i,\mathrm{spec}}^{(r)}:=2\gamma_r\|\mathbf R_r\|_2d_i^{(r)}. \label{eq:specradius}
\end{equation}
Equivalently,
\[
 g_i^{(r)}
 \in
 \left[
 \widehat q_i^{(r)}-\rho_{i,\mathrm{spec}}^{(r)},
 \widehat q_i^{(r)}+\rho_{i,\mathrm{spec}}^{(r)}
 \right].
\]
\end{lemma}
\begin{proof}
Because $\mathbf M_r=\sum_j(y_j^{(r)}-\bar y_r)\boldsymbol u_j^{(r)}(\boldsymbol u_j^{(r)})^\top$,
\begin{align*}
 \sum_{j\ne i}\left((\boldsymbol u_i^{(r)})^\top\boldsymbol u_j^{(r)}\right)^2(y_j^{(r)}-\bar y_r)
 &=(\boldsymbol u_i^{(r)})^\top \mathbf M_r\boldsymbol u_i^{(r)}-a_i^{(r)}(y_i^{(r)}-\bar y_r). 
\end{align*}
Insert this identity into \eqref{eq:grad}, substitute $\mathbf M_r=\mu_r\mathbf I+\mathbf R_r$, and use $(\boldsymbol u_i^{(r)})^\top \boldsymbol u_i^{(r)}=d_i^{(r)}$ to obtain \eqref{eq:specdecomp}--\eqref{eq:qhat}. The Rayleigh bound $|\boldsymbol u^\top\mathbf R_r\boldsymbol u|\le\|\mathbf R_r\|_2\|\boldsymbol u\|_2^2$ gives \eqref{eq:specradius}.
\end{proof}
Thus Lemma~\ref{lem:spectral} provides a directly computable safe interval for every current gradient coordinate. Both the mean-path and spectral intervals are independently valid, so we use their intersection:
\begin{align}
 \ell_i^{(r)}&=\max\{0,q_{i,\mathrm{mean}}^{(r)}-\rho_{i,\mathrm{mean}}^{(r)},\widehat q_i^{(r)}-\rho_{i,\mathrm{spec}}^{(r)}\},\nonumber\\
 u_i^{(r)}&=\min\{q_{i,\mathrm{mean}}^{(r)}+\rho_{i,\mathrm{mean}}^{(r)},\widehat q_i^{(r)}+\rho_{i,\mathrm{spec}}^{(r)}\}. \label{eq:intersect}
\end{align}
For a tagged element $e=(r,i)$, we use the shorthand
\[
\phi_e:=\phi_i^{(r)},\qquad
\ell_e:=\ell_i^{(r)},\qquad
u_e:=u_i^{(r)},
\]
and, when referring to the spectral interval alone,
\[
\widehat q_e:=\widehat q_i^{(r)},\qquad
\rho_e:=\rho_{i,\mathrm{spec}}^{(r)}.
\]
The intersected interval is therefore no wider than either input interval. Its construction requires one eigenvalue computation of a symmetric $K_r\times K_r$ matrix per mode and no $N_r\times N_r$ pairwise table. At the mean state $\boldsymbol y^{(r)}=\bar y_r\boldsymbol 1$, $\mathbf M_r=\mathbf R_r=0$, so the spectral interval collapses to the exact mean-state gradient.

\subsection{Certified direction and selective exact queries}
The linear maximization oracle (LMO) chooses the feasible base with the largest current aggregate gradient. For a feasible removal base $\mathcal A\in\cB_B$, define the optimistic coordinate weights
\begin{equation}
 z_e(\mathcal A)=\begin{cases}\ell_e,&e\in\mathcal A,\\u_e,&e\notin\mathcal A,\end{cases}\qquad
 \mathcal C^+(\mathcal A)\in\arg\max_{\mathcal C\in\cB_B}\sum_{e\in\mathcal C}z_e(\mathcal A), \label{eq:z}
\end{equation}
and the certificate gap
\begin{equation}
 \epsilon_{\mathcal A}=\sum_{e\in\mathcal C^+(\mathcal A)}z_e(\mathcal A)-\sum_{e\in\mathcal A}\ell_e. \label{eq:epsA}
\end{equation}
Uncertainty on coordinates common to both bases cancels.

\begin{lemma}[Certificate-gap bound]\label{lem:lmo-gap}
For every feasible $\mathcal C$,
\begin{equation*}
 g(\mathcal C)-g(\mathcal A)\le\sum_{e\in \mathcal C\setminus\mathcal A}u_e-\sum_{e\in\mathcal A\setminus\mathcal C}\ell_e\le\epsilon_{\mathcal A}.
\end{equation*}
\end{lemma}
\begin{proof}
Apply the coordinatewise bounds $g_e\le u_e$ and $g_e\ge\ell_e$ to the elements in the symmetric difference of $\mathcal C$ and $\mathcal A$, then maximize the resulting upper bound over feasible bases using the definition of $\mathcal C^+(\mathcal A)$.
\end{proof}

$g(\mathcal A)$ can be evaluated without querying all $B$ coordinate gradients. Define
\begin{equation*}
 \mathbf S_{\mathcal A,r}=\sum_{i\in \mathcal A_r}\boldsymbol u_i^{(r)}(\boldsymbol u_i^{(r)})^\top
 =\mathbf G_r^{\rm full}-\mathbf U_r(\mathcal N_r\setminus\mathcal A_r)^\top \mathbf U_r(\mathcal N_r\setminus\mathcal A_r).
\end{equation*}
Summing \eqref{eq:grad} over $\mathcal A_r$ gives
\begin{equation}
 g(\mathcal A)=\sum_r\gamma_r\!\left[\sum_{i\in \mathcal A_r}a_i^{(r)}+2\langle \mathbf G_r,\mathbf S_{\mathcal A,r}\rangle_F-2\sum_{i\in \mathcal A_r}a_i^{(r)}y_i^{(r)}\right]. \label{eq:aggregate}
\end{equation}
For fixed retained budget $L$, \eqref{eq:aggregate} uses only modal Grams, scalar sums, and at most $L$ retained directional rows.

\begin{lemma}[Certified multiplicative oracle]\label{lem:tau}
For a candidate $\mathcal A$ with certificate gap $\epsilon_{\mathcal A}$,
\begin{equation}
 g(\mathcal A)\le
 \max_{\mathcal C\in\cB_B}g(\mathcal C)
 \le g(\mathcal A)+\epsilon_{\mathcal A}. \label{eq:Mupper}
\end{equation}
Define
\begin{equation}
 \tau_{\mathcal A}=
 \begin{cases}
 \dfrac{g(\mathcal A)}{g(\mathcal A)+\epsilon_{\mathcal A}},
 & g(\mathcal A)+\epsilon_{\mathcal A}>0,\\[1ex]
 1, & \text{otherwise}.
 \end{cases}
 \label{eq:tau}
\end{equation}
Then
\begin{equation}
 g(\mathcal A)\ge\tau_{\mathcal A}\max_{\mathcal C\in\cB_B}g(\mathcal C). \label{eq:multlmo}
\end{equation}
\end{lemma}
\begin{proof}
Lemma~\ref{lem:lmo-gap} applied to an exact maximizing base gives \eqref{eq:Mupper}. Rearranging yields \eqref{eq:multlmo}.
\end{proof}
Thus $\tau_{\mathcal A}$ is a computable multiplicative certificate for the fraction of the exact LMO value achieved by the current base. It is used as the stopping criterion for selective exact-gradient refinement.

\begin{corollary}[Margin-separated zero-query LMO]\label{cor:zeroquery}
Let $\mathcal A\in\cB_B$ be a feasible base. If
\begin{equation}
 \ell_e\ge u_f
 \qquad\text{for every feasible exchange }\mathcal A-e+f,
 \quad e\in\mathcal A,\ f\notin\mathcal A, \label{eq:zeroquerysep}
\end{equation}
then $\mathcal A$ is an exact maximum-weight base for the true gradient, $\epsilon_{\mathcal A}=0$, and $\tau_{\mathcal A}=1$.
\end{corollary}
\begin{proof}
Condition~\eqref{eq:zeroquerysep} implies that every feasible one-element exchange from $\mathcal A$ has nonpositive true-gradient gain. By the matroid exchange characterization of maximum-weight bases, $\mathcal A$ maximizes the exact gradient. The same inequalities imply that the optimistic competitor $\mathcal C^+(\mathcal A)$ cannot improve on $\mathcal A$, so \eqref{eq:epsA} gives $\epsilon_{\mathcal A}=0$ and Lemma~\ref{lem:tau} gives $\tau_{\mathcal A}=1$.
\end{proof}
Corollary~\ref{cor:zeroquery} certifies the exact LMO without any exact coordinate query. For symmetric intervals $[\widehat q_e-\rho_e,\widehat q_e+\rho_e]$, it is sufficient that $\widehat q_e-\widehat q_f\ge\rho_e+\rho_f$ on every feasible boundary exchange.

The resulting stopping rule is scale free: exact coordinate gradients are queried only until $\tau_{\mathcal A}\ge\tau_0$, with $\tau_0=0.9999$ in all reported experiments. Corollary~\ref{cor:zeroquery} identifies the zero-query regime explicitly. Otherwise, exact evaluations are confined to unresolved exchange candidates encountered during recertification. Within such an unresolved set, the implementation prioritizes coordinates by
\begin{equation}
 (u_e-\ell_e)\left(1+\frac{\phi_e}{\max_f\phi_f}\right), \label{eq:querypriority}
\end{equation}
in descending order. If the active unresolved set is $\mathcal U_t$, one recertification queries
\begin{equation}
 b_t=\min\bigl\{|\mathcal U_t|,\max(2,\lceil|\mathcal U_t|/4\rceil)\bigr\}
 \label{eq:querybatch}
\end{equation}
coordinates. This priority affects only which safe intervals are collapsed first. The certificate remains valid for any query order. The resulting saving is instance dependent, consistent with the absence of a general sublinear value-oracle guarantee \cite{peng2025lower}.

For the implementation, define the clipped spectral proxy
\[
q_e=\min\{u_e,\max\{\widehat q_e,\ell_e\}\},
\]
and let
\[
\operatorname{Base}(\boldsymbol w)
\in
\arg\max_{\mathcal C\in\cB_B}
\sum_{e\in\mathcal C}w_e.
\]
At each outer step the initial candidate is $\mathcal A=\operatorname{Base}(\boldsymbol q)$. After each exact-query batch, we retain the member of
\[
\left\{\mathcal A,\,\operatorname{Base}(\boldsymbol q),\,\operatorname{Base}(\boldsymbol\ell)\right\}
\]
with the largest exact aggregate derivative \eqref{eq:aggregate}. All calls to $\operatorname{Base}$ are exact maximum-weight truncated-partition-matroid linear oracles.

\subsection{Residual-controlled adaptive steps and algorithm}

\subsubsection{Gram-only state and directional-polynomial updates}
At adaptive step $t$, let $\boldsymbol x_t$ be the current fractional state and let $\boldsymbol 1_{\mathcal C}\in\{0,1\}^{|\cE|}$ denote the incidence vector of a removal set $\mathcal C$. For the accepted base $\mathcal C_t$, define $\boldsymbol h_t:=\boldsymbol 1_{\mathcal C_t}$.
\begin{equation}
\begin{aligned}
\mathbf S_{r,t}
&:=
\sum_{i\in\mathcal C_{t,r}}
\boldsymbol u_i^{(r)}(\boldsymbol u_i^{(r)})^\top,\\
P_{r,t}
&:=
\sum_{i\in\mathcal C_{t,r}}a_i^{(r)}
+2\left\langle\mathbf G_r(\boldsymbol x_t),\mathbf S_{r,t}\right\rangle_F
-2\sum_{i\in\mathcal C_{t,r}}a_i^{(r)}\bigl(1-x_{t,i}^{(r)}\bigr),\\
Q_{r,t}
&:=
\|\mathbf S_{r,t}\|_F^2
-\sum_{i\in\mathcal C_{t,r}}a_i^{(r)}
\ge0,\\
H_r(\boldsymbol x_t+s\boldsymbol h_t)
&=
H_r(\boldsymbol x_t)-sP_{r,t}+s^2Q_{r,t}.
\end{aligned}
\label{eq:poly}
\end{equation}
Set $D_t:=g(\mathcal C_t)$ evaluated at $\boldsymbol x_t$. The same directional derivative is obtained from the polynomial \eqref{eq:poly} and agrees analytically with the aggregate LMO derivative in \eqref{eq:aggregate}. This equality is verified at every adaptive step in the implementation.

\subsubsection{Residual-controlled step selection}
At adaptive step $t$, let $s_{\rm rem}$ be the current remaining path length and set $V_t:=\widetilde v(\boldsymbol x_t)$. For $0\le s\le s_{\rm rem}$, define
\begin{equation}
 \mathcal R_t(s)=sD_t-[\widetilde v(\boldsymbol x_t+s \boldsymbol h_t)-V_t]. \label{eq:residual}
\end{equation}
Monotone diminishing-returns (DR) submodularity implies concavity of $\widetilde v(\boldsymbol x_t+s \boldsymbol h_t)$ in the nonnegative direction $\boldsymbol h_t$, hence $\mathcal R_t(s)\ge0$ and the normalized residual $\mathcal R_t(s)/(sD_t)$ is nondecreasing when $D_t>0$. Because \eqref{eq:poly} makes the objective along the line a product of explicit quadratics, the largest step obeying
\begin{equation}
 \mathcal R_t(s_t)\le\eta s_tD_t,\qquad 0\le\eta<1, \label{eq:rescontrol}
\end{equation}
is obtained by a scalar monotone search without additional coordinate gradients. We use a single $\eta=0.25$ in the reported experiments.

\begin{algorithm}[t]
\caption{Spectral-Residual Continuous Greedy (SR-CG)}
\label{alg:srcg-full}
\begin{algorithmic}[1]
\Require Modal factors $\{\mathbf U_r\}$, retained budget $L$, thresholds $\tau_0$ and $\eta$
\State Precompute $\mathbf G_r^{\rm full}=\mathbf U_r^\top\mathbf U_r$, $a_i^{(r)}$, $b_i^{(r)}$, $Z_r$, and $\phi_i^{(r)}$
\State $\boldsymbol x\gets\boldsymbol0$, $\boldsymbol y\gets\boldsymbol1$, $\mathbf G_r\gets\mathbf G_r^{\rm full}$, $s_{\rm rem}\gets1$
\While{$s_{\rm rem}>0$}
  \State Compute $[\ell_e,u_e]$ by \eqref{eq:intersect}, then set $q_e\gets\min\{u_e,\max\{\widehat q_e,\ell_e\}\}$ and $\mathcal A\gets\operatorname{Base}(\boldsymbol q)$
  \State Compute $\mathcal C^+(\mathcal A)$ by \eqref{eq:z}, $\epsilon_{\mathcal A}$ by \eqref{eq:epsA}, and $\tau_{\mathcal A}$ by \eqref{eq:tau}, using $g(\mathcal A)$ from \eqref{eq:aggregate}
  \While{$\tau_{\mathcal A}<\tau_0$}
    \State Set $\mathcal U_t\gets(\mathcal A\triangle\mathcal C^+(\mathcal A))\cap\{e:\ell_e<u_e\}$. If $\mathcal U_t$ is empty, use all unresolved coordinates
    \State Compute priorities by \eqref{eq:querypriority} and $b_t$ by \eqref{eq:querybatch}
    \State For the $b_t$ highest-priority $e=(r,i)\in\mathcal U_t$, evaluate $g_i^{(r)}(\boldsymbol x)$ by \eqref{eq:grad} and set $\ell_e\gets g_i^{(r)}(\boldsymbol x)$, $u_e\gets g_i^{(r)}(\boldsymbol x)$
    \State Set $\mathcal A\gets\arg\max_{\mathcal C\in\{\mathcal A,\operatorname{Base}(\boldsymbol\ell),\operatorname{Base}(\boldsymbol q)\}}g(\mathcal C)$ using \eqref{eq:aggregate}
    \State Recompute $\mathcal C^+(\mathcal A)$, $\epsilon_{\mathcal A}$, and $\tau_{\mathcal A}$ using \eqref{eq:z}--\eqref{eq:tau}
  \EndWhile
  \State Set $\mathcal C_t\gets\mathcal A$ and form $\mathbf S_{r,t},P_{r,t},Q_{r,t}$ by \eqref{eq:poly}
  \State Choose the largest $s_t\le s_{\rm rem}$ satisfying \eqref{eq:rescontrol} by the scalar monotone search for \eqref{eq:residual}
  \State Update $\boldsymbol x\gets\boldsymbol x+s_t\boldsymbol1_{\mathcal C_t}$, $\mathbf G_r\gets\mathbf G_r-s_t\mathbf S_{r,t}$, and $\boldsymbol y\gets\boldsymbol1-\boldsymbol x$, then set $s_{\rm rem}\gets s_{\rm rem}-s_t$
\EndWhile
\State While $\boldsymbol x$ is fractional, choose a feasible pairwise pipage exchange $(i,j)$ and endpoints $\lambda_-,\lambda_+$
\State Set $\lambda^\star\in\arg\max_{\lambda\in\{\lambda_-,\lambda_+\}}\widetilde v(\boldsymbol x+\lambda(\boldsymbol e_i-\boldsymbol e_j))$ and update to that endpoint, then repeat until $\mathcal C^0$ is obtained \cite{chekuri2010dependent} (Lemma~\ref{lem:round})
\State Set $\mathcal C\gets\mathcal C^0$, let $\mathcal P(\mathcal C)$ contain the $\min(2L,|\mathcal C|)$ removed elements with smallest lexicographic $(x_{T,e},\phi_e)$, and let $\mathcal K(\mathcal C)=\cE\setminus\mathcal C$ be ordered decreasingly by $(x_{T,e},\phi_e)$
\State Find $(i^\star,j^\star)\in\arg\min_{i\in\mathcal P(\mathcal C),\,j\in\mathcal K(\mathcal C):\,\mathcal C-i+j\in\cB_B}F(\mathcal L(\mathcal C-i+j))$
\State If it strictly lowers $F(\mathcal L(\mathcal C))$, set $\mathcal C\gets\mathcal C-i^\star+j^\star$, recompute $\mathcal P(\mathcal C)$ and $\mathcal K(\mathcal C)$, and repeat. Otherwise, terminate
\State \Return $\mathcal C_{\rm out}\gets\mathcal C$
\end{algorithmic}
\end{algorithm}

\section{Guarantees and complexity}
The analysis separates three effects: inexact direction certification, finite-step curvature loss, and discrete completion. We first state the variable-step guarantee, then give the rounding argument, computational cost, an instance-level FP lower certificate, and the relation between FP and least-squares reconstruction.

\subsection{Finite-step approximation guarantee}
The adaptive rule incorporates the realized residual directly into the optimization dynamics.
Let $T$ denote the number of accepted adaptive steps. For each $t$, let $g_t(\mathcal C)$ denote the aggregate gradient $g(\mathcal C)$ evaluated at $\boldsymbol x_t$, set
\[
\tau_t:=\tau_{\mathcal C_t},
\qquad
M_t:=\max_{\mathcal C\in\cB_B}g_t(\mathcal C).
\]

\begin{lemma}[Exact variable-step recurrence]\label{lem:finite-recurrence}
Suppose $D_t\ge\tau_tM_t$ and $\sum_{t=0}^{T-1}s_t=1$. For any optimal feasible base $\mathcal C^\star$,
\begin{equation}
 \widetilde v(\boldsymbol x_T)\ge
 \left[1-\prod_{t=0}^{T-1}(1-s_t\tau_t)\right]v(\mathcal C^\star)
 -\sum_{t=0}^{T-1}\mathcal R_t(s_t)\prod_{j=t+1}^{T-1}(1-s_j\tau_j). \label{eq:finite}
\end{equation}
\end{lemma}
\begin{proof}
Let $\boldsymbol g_t=\nabla\widetilde v(\boldsymbol x_t)$ and write $g_t(\mathcal C)=\boldsymbol g_t^\top\boldsymbol 1_{\mathcal C}$. Monotonicity gives
$\widetilde v(\boldsymbol x_t\vee\boldsymbol 1_{\mathcal C^\star})\ge\widetilde v(\boldsymbol 1_{\mathcal C^\star})=v(\mathcal C^\star)$. Since $\widetilde v$ is DR-submodular, concavity along nonnegative directions yields
\[
\widetilde v(\boldsymbol x_t\vee\boldsymbol 1_{\mathcal C^\star})-V_t
\le \boldsymbol g_t^\top[(\boldsymbol x_t\vee\boldsymbol 1_{\mathcal C^\star})-\boldsymbol x_t]
\le g_t(\mathcal C^\star)\le M_t,
\]
where the second inequality uses $\boldsymbol g_t\ge\boldsymbol 0$ coordinatewise. Therefore
$M_t\ge v(\mathcal C^\star)-V_t$, and the certified direction satisfies
$D_t\ge\tau_t[v(\mathcal C^\star)-V_t]$. Combining this with the exact residual identity
$V_{t+1}=V_t+s_tD_t-\mathcal R_t(s_t)$ gives
\[
V_{t+1}\ge(1-s_t\tau_t)V_t+s_t\tau_t v(\mathcal C^\star)-\mathcal R_t(s_t).
\]
Unrolling the scalar recurrence yields \eqref{eq:finite}.
\end{proof}

\begin{lemma}[Residual-absorbed product bound]\label{lem:controlled}
Under the conditions of Lemma~\ref{lem:finite-recurrence}, let $r_t=\mathcal R_t(s_t)/(s_tD_t)$, with $r_t=0$ if $s_tD_t=0$. Then $0\le r_t\le1$ and
\begin{equation}
 \widetilde v(\boldsymbol x_T)\ge
 \left[1-\prod_t\{1-(1-r_t)s_t\tau_t\}\right]v(\mathcal C^\star). \label{eq:controlledreal}
\end{equation}
If, in addition, every accepted step satisfies $\mathcal R_t(s_t)\le\eta s_tD_t$, then
\begin{equation}
 \widetilde v(\boldsymbol x_T)\ge
 \left[1-\prod_t\{1-(1-\eta)s_t\tau_t\}\right]v(\mathcal C^\star). \label{eq:controlledeta}
\end{equation}
\end{lemma}
\begin{proof}
Monotonicity along $\boldsymbol h_t\ge0$ and $\mathcal R_t(s_t)\ge0$ give $0\le r_t\le1$. Equation \eqref{eq:residual} then gives
$V_{t+1}-V_t=(1-r_t)s_tD_t$. Because $0\le r_t\le1$, the LMO bound preserves the inequality direction and yields
$V_{t+1}-V_t\ge(1-r_t)s_t\tau_t[v(\mathcal C^\star)-V_t]$.
Unrolling the multiplicative gap contraction gives \eqref{eq:controlledreal}. Under the stated residual-control condition, $r_t\le\eta$ for every accepted step, which gives \eqref{eq:controlledeta}.
\end{proof}
The second bound is available a priori from the prescribed residual threshold, whereas the first uses the tighter realized residual ratios. We therefore use the second form to derive the uniform finite-step guarantee below.

\begin{theorem}[Finite-step approximation guarantee]
\label{thm:finite}
Suppose that every accepted direction satisfies $\tau_t\ge\tau_0$ and every adaptive step satisfies
$\mathcal R_t(s_t)\le\eta s_tD_t$, with $\sum_t s_t=1$. Then
\begin{equation*}
\widetilde v(\boldsymbol x_T)
\ge
\left[1-e^{-(1-\eta)\tau_0}\right]
v(\mathcal C^\star).
\end{equation*}
\end{theorem}
\begin{proof}
From \eqref{eq:controlledeta} and $1-z\le e^{-z}$,
\[
 \prod_t\{1-(1-\eta)s_t\tau_t\}
 \le \exp\!\left(-(1-\eta)\sum_ts_t\tau_t\right)
 \le e^{-(1-\eta)\tau_0},
\]
where the last inequality uses $\sum_ts_t=1$.
\end{proof}
Theorem~\ref{thm:finite} separates the two controls: $\tau_0$ governs linear-oracle certification, whereas $\eta$ governs finite-step curvature loss. As $\tau_0\to1$ and $\eta\to0$, the approximation factor approaches the classical $1-1/e$ guarantee.

\subsection{Deterministic rounding and path-guided exchange}
The final fractional point is a convex combination of feasible bases,
\begin{equation*}
 \boldsymbol x_T=\sum_ts_t\boldsymbol 1_{\mathcal C_t},\qquad s_t\ge0,\quad \sum_ts_t=1.
\end{equation*}
This representation allows a deterministic sequence of feasible pipage exchanges to convert the fractional point into a discrete base without weakening the objective bound.

\begin{lemma}[Deterministic discrete preservation]\label{lem:round}
There is a deterministic sequence of feasible pairwise pipage endpoint moves producing a base $\mathcal C^0$ such that
\begin{equation}
 v(\mathcal C^0)\ge \widetilde v(\boldsymbol x_T). \label{eq:rounddet}
\end{equation}
\end{lemma}
\begin{proof}
For a multilinear extension of a submodular function, the restriction to a feasible exchange line $\boldsymbol x+\lambda(\boldsymbol e_i-\boldsymbol e_j)$ is convex because its second derivative is $-2\partial_{ij}\widetilde v\ge0$. At each pairwise pipage step at least one feasible endpoint has objective no smaller than the current fractional point. Selecting the better endpoint preserves feasibility and does not decrease $\widetilde v$. Repeating until an integral point is reached proves \eqref{eq:rounddet}.
\end{proof}
Consequently, the lower bounds from
Lemmas~\ref{lem:finite-recurrence} and~\ref{lem:controlled} also hold
pointwise for $\mathcal C^0$.
Within SR-CG, PGX is the final discrete refinement stage following
deterministic rounding. It couples the rounded solution to the
preceding fractional continuous stage through the final fractional
state $\boldsymbol{x}_T$.
From
$\boldsymbol x_T=\sum_t s_t\boldsymbol 1_{\mathcal C_t}$, each
coordinate satisfies
$x_{T,e}=\sum_{t:e\in\mathcal C_t}s_t$ and therefore summarizes the
total path weight of the accepted SR-CG directions containing $e$.

For the current removal base $\mathcal C$, let
$\mathcal P(\mathcal C)$ contain the at most $2L$ elements of
$\mathcal C$ with smallest lexicographic scores
$(x_{T,e},\phi_e)$, and let
$\mathcal K(\mathcal C)=\mathcal E\setminus\mathcal C$ be the $L$
retained tagged elements, ordered by decreasing
$(x_{T,e},\phi_e)$. Here $x_{T,e}$ is the primary ordering, whereas
$\phi_e$ is only a deterministic secondary ordering. Among feasible
exchanges $\mathcal C-i+j$ with
$i\in\mathcal P(\mathcal C)$ and $j\in\mathcal K(\mathcal C)$, PGX
evaluates the exact retained FP and accepts the best strict
improvement. The candidate set is recomputed after each accepted swap,
and the process terminates when no improving candidate exchange
remains.

The distinguishing design choice of the PGX stage within SR-CG is this
state-to-frontier coupling: $\boldsymbol{x}_T$ restricts and prioritizes
the exchange search, while the exact FP---not $x_{T,e}$ or
$\phi_e$---determines whether a swap is accepted.
The factor $2L$ is
fixed for all experiments. It bounds the removal frontier independently
of $N_\Sigma$ and is not tuned by the data set. Since
$|\mathcal K(\mathcal C)|=L$, one scan contains at most $2L^2$
candidate pairs before feasibility filtering.

\begin{lemma}[PGX monotonicity]\label{lem:pgx}
Let $\mathcal C_{\rm out}$ be the base returned by PGX initialized at $\mathcal C^0$. Then
\begin{equation}
 F(\mathcal L(\mathcal C_{\rm out}))\le F(\mathcal L(\mathcal C^0)),\qquad
 v(\mathcal C_{\rm out})\ge v(\mathcal C^0)\ge\widetilde v(\boldsymbol x_T). \label{eq:pgxmono}
\end{equation}
\end{lemma}
\begin{proof}
Every accepted exchange is feasible and is accepted only after its exact retained FP is verified to be strictly smaller than the current FP. Thus FP decreases monotonically, equivalently $v$ increases. Combining this fact with Lemma~\ref{lem:round} gives \eqref{eq:pgxmono}.
\end{proof}
Hence every lower bound established for the deterministically rounded output remains valid for the final SR-CG output.
Combining Theorem~\ref{thm:finite} with Lemmas~\ref{lem:round} and~\ref{lem:pgx} shows that the same finite-step approximation guarantee holds for the final SR-CG output.

\begin{corollary}[Guarantee in the original FP scale]\label{cor:rawfp}
Let $\alpha=1-e^{-(1-\eta)\tau_0}$ and let $F^*$ denote the minimum feasible retained FP. Under the conditions of Theorem~\ref{thm:finite}, define the final retained design by $\mathcal L_{\rm out}=\mathcal L(\mathcal C_{\rm out})$. Then
\begin{equation}
 F(\mathcal L_{\rm out})\le (1-\alpha)F(\mathcal N)+\alpha F^*. \label{eq:rawfp}
\end{equation}
\end{corollary}
\begin{proof}
Theorem~\ref{thm:finite} together with Lemmas~\ref{lem:round} and~\ref{lem:pgx} gives
$1-F(\mathcal L_{\rm out})/F(\mathcal N)\ge\alpha[1-F^*/F(\mathcal N)]$. Rearranging yields \eqref{eq:rawfp}.
\end{proof}
Equivalently, SR-CG captures at least an $\alpha$ fraction of the optimal FP reduction from the full design. The finishing stage is deliberately separated from the certified gradient accounting: it adds no exact multilinear-gradient coordinates. A control experiment applies unrestricted exact one-exchange refinement to other initializations as well. This confirms that local exchange can benefit multiple starting points. Therefore, the experiments report both the deterministically rounded intermediate output and the complete SR-CG output rather than attributing all post-rounding gains to the continuous path.

\subsection{Complexity and implementation}
Computing and caching the full modal Grams and row statistics once costs
\begin{equation*}
 O\!\left(\sum_rN_rK_r^2\right)
\end{equation*}
and requires no $N_r\times N_r$ pairwise table. At adaptive step $t$, if $q_t$ coordinates are evaluated exactly and $J_t$ relative-base recertifications are required, the cost is
\begin{equation*}
 O\!\left(\sum_rK_r^3+(J_t+1)N_\Sigma+(q_t+L)(\max_rK_r)^2\right).
\end{equation*}
The residual search itself repeatedly evaluates only the explicit low-dimensional directional polynomial and so adds no exact-coordinate queries. Across the adaptive path, the total exact-coordinate work is $\sum_tq_t$. A full-gradient reference evaluates $N_\Sigma$ coordinates per step. The reported reference uses 36 such steps. The scan and recertification terms remain linear in $N_\Sigma$. PGX is a separate discrete cost: selecting its candidate set is $O(N_\Sigma\log L)$, and one pass tests at most $2L^2$ exchanges. With cached modal Grams, each test updates at most two $K_r\times K_r$ modal factors, giving $O(L^2(\max_rK_r)^2)$ arithmetic per pass. The number of PGX scans is at most one plus the number of accepted exchanges. The computational benefit lies in reducing exact gradient evaluations during the continuous stage, while PGX adds separate discrete FP swap evaluations. The latter improve the final discrete design but are not counted as continuous-stage gradient evaluations.

\subsection{Additional FP and reconstruction analysis}
\subsubsection{Energy-aware FP certificate}
Although Algorithm~\ref{alg:srcg-full} does not branch on an optimality certificate, an instance-level lower bound is useful for validating designs and remains valid for arbitrary row energies. Recall $d_i^{(r)}=\|\boldsymbol u_i^{(r)}\|_2^2$ from \eqref{eq:ybar}. Let $d_{r,(1)}\le\cdots\le d_{r,(N_r)}$ be the ordered mode-$r$ row energies, and set
\begin{equation*}
 s_r(\ell)=\sum_{j=1}^{\ell}d_{r,(j)}.
\end{equation*}
For a retained set $\mathcal L_r$ of cardinality $\ell_r$, let $\mathbf Q_r=\mathbf U_r(\mathcal L_r)^{\top}\mathbf U_r(\mathcal L_r)$. By Cauchy--Schwarz on the $K_r$ eigenvalues,
\begin{equation*}
 \operatorname{tr}(\mathbf Q_r^2)\ge\frac{\operatorname{tr}(\mathbf Q_r)^2}{K_r}.
\end{equation*}
Moreover $\operatorname{tr}(\mathbf Q_r)$ is the sum of the retained row energies and is at least $s_r(\ell_r)$. Therefore
\begin{equation}
 F_r(\mathcal L_r)\ge\frac{s_r(\ell_r)^2}{K_r}. \label{eq:modeLB}
\end{equation}

\begin{theorem}[Energy-aware frame lower bound]\label{thm:energy}
Let $F^*$ be the minimum tensor FP over all designs satisfying \eqref{eq:budget}. Then
\begin{equation}
 F^*\ge F_E:=\min_{\substack{\ell_r\ge\kappa_r\\\sum_r\ell_r=L}}
 \prod_{r=1}^R\frac{s_r(\ell_r)^2}{K_r}. \label{eq:FE}
\end{equation}
If $F_E>0$, every feasible design $\widehat{\mathcal L}$ has the valid certificate
\begin{equation}
 \frac{F(\widehat{\mathcal L})}{F^*}\le\frac{F(\widehat{\mathcal L})}{F_E}. \label{eq:Fcert}
\end{equation}
\end{theorem}
\begin{proof}
For any feasible modal cardinalities $(\ell_1,\ldots,\ell_R)$, \eqref{eq:modeLB} and the Kronecker FP factorization give
$F(\mathcal L)\ge\prod_r s_r(\ell_r)^2/K_r$. Minimizing this lower bound over all feasible modal cardinalities yields $F^*\ge F_E$. If $F_E>0$, then $F^*\ge F_E$ implies \eqref{eq:Fcert}.
\end{proof}
The minimization in \eqref{eq:FE} is a small dynamic program over mode cardinalities. For unit-norm rows, $s_r(\ell)=\ell$ and the bound reduces to the familiar Welch-type form $\ell_r^2/K_r$ mode by mode.

\subsubsection{Reconstruction interpretation and rank conditions}
FP is a spectral concentration measure, not the least-squares risk itself. For a retained modal Gram $\mathbf Q_r\succ0$, the tensor least-squares covariance obeys
\begin{equation*}
 \operatorname{tr}\!\left[\left(\bigotimes_r\mathbf Q_r\right)^{-1}\right]
 =\prod_r\operatorname{tr}(\mathbf Q_r^{-1}).
\end{equation*}
A small FP encourages balanced spectra but does not, without a lower spectral bound, imply a universal multiplicative MSE ratio. For example, a spectrum $(1,\epsilon)$ can have moderate squared eigenvalue sum while its inverse trace diverges as $\epsilon\to0$. We therefore keep FP guarantees and reconstruction measurements separate.

The cardinality condition $|\mathcal L_r|\ge K_r$ is necessary but not sufficient for $\mathbf U_r(\mathcal L_r)$ to have full column rank. All reported reconstruction designs are explicitly rank-checked. The handwritten-digit protocol uses training data only to estimate the spatial factors. The held-out test images do not affect the sampling design.

\section{Experiments}
We evaluate SR-CG on generic synthetic tensors, a controlled cross-mode
allocation stress family, and handwritten-digit reconstruction.
For stage-wise evaluation, Rounded SR-CG denotes the intermediate
output after deterministic rounding and before PGX, whereas SR-CG
denotes the complete algorithm including the PGX stage.
Fixed-step continuous greedy (Fixed-step CG)
uses a fixed step size. Exact-coordinate evaluations count exact
multilinear-gradient evaluations in the continuous stage. PGX's discrete FP
swap evaluations are reported separately. MSE-greedy is used only as an exact
least-squares reference in the synthetic reconstruction comparison. We use
least-squares (LS) MSE in the synthetic stress test and normalized mean-squared error (NMSE) for handwritten-digit
reconstruction.

\subsection{Experimental protocol and validation}
Greedy-FP uses the current exact feasible FP marginal at every removal, and FFW uses its published midpoint multilinear-FP scores and mode-wise normalization. Fixed-step CG uses $T=36$ uniform steps, and its fractional path is evaluated with eight randomized swap-rounding repetitions, matching the implementation. The internal MSE-greedy reference uses exact one-row-at-a-time least-squares MSE updates with rank-one inverse downdates and is not treated as a generic FP baseline.

For Greedy-FP and FFW, the original methods use the rank-based lower quota $|\mathcal L_r|\ge K_r$. When an experiment imposes a stricter common quota $\kappa_r>K_r$, we change only this feasibility quota to $|\mathcal L_r|\ge\kappa_r$ for all methods. Their scoring rules are unchanged. The final modal cardinalities remain free for all methods. The common quota changes only the feasible set and does not prescribe a fixed allocation. All synthetic comparisons use paired seeds. SR-CG uses $\tau_0=0.9999$ and $\eta=0.25$ unless explicitly varied. After deterministic pipage-style rounding, PGX uses the same fixed frontier factor $2L$ in every experiment. No data-dependent PGX parameter is tuned. Exact-gradient-coordinate counts refer to the continuous stage. PGX's exact discrete FP swap evaluations are reported separately.

All reported reconstruction designs are checked for full column rank before least-squares reconstruction. Random seeds and raw per-trial results are exported by the experiment code. No parameter is retuned separately for the generic, stress, or handwritten-digit experiments except when the parameter itself is the object of an ablation.

The exact audit enumerates the optimum on 120 small instances and finds zero violations of the FP/Shapley identities, spectral intervals, multiplicative LMO certificate, residual nonnegativity, adaptive residual rule, exact variable-step recurrence, residual-absorbed product certificate, or deterministic rounding. The minimum observed slack is $0.3682$ for Lemma~\ref{lem:finite-recurrence} and $0.3068$ for Lemma~\ref{lem:controlled}. The minimum rounding gain is zero. The implementation validation additionally finds no discrepancy for Greedy-FP, the published tensor FFW rule, the truncated-partition-matroid linear oracle, or MSE-greedy, and no PGX FP increase or feasibility violation in the tested cases.

\subsection{Generic tensors}
The generic benchmark contains 16 paired three-mode instances. The candidate dimensions are $(N_1,N_2,N_3)=(80,90,100)$, the latent dimensions are $(K_1,K_2,K_3)=(6,7,8)$, the minimum retained counts are $(10,11,12)$, and the total retained budget is $L=42$.

Each modal factor is generated from clustered unit-norm rows. For mode $r$, random unit cluster centers are first generated. Each row then combines its assigned center with an independent random unit direction,
\[
\boldsymbol z_i^{(r)}
=
\sqrt{\rho_r}\,\boldsymbol c_i^{(r)}
+
\sqrt{1-\rho_r}\,\boldsymbol\xi_i^{(r)}
+
10^{-4}\boldsymbol\nu_i^{(r)},
\]
and is then normalized to unit Euclidean norm. The three modes use $\rho_r=(0.81,0.56,0.31)$, producing different levels of directional redundancy across modes. The small Gaussian perturbation avoids degenerate repeated rows.

Relative to Fixed-step CG, Rounded SR-CG reduces mean FP/Greedy-FP from $1.388$ to $1.141$, the mean number of continuous-greedy steps from $36$ to $16$, and exact-coordinate evaluations from $4375$ to $2607$. The lower query count is accompanied by better FP after rounding. Spectral intervals avoid exact evaluation once the relevant coordinate ordering is certified, while residual control avoids rebuilding a direction when the exact directional polynomial shows that it remains accurate over a longer step. The two mechanisms remove different sources of repeated continuous-stage work.

PGX does not add multilinear-gradient queries. Starting from Rounded SR-CG, it lowers the paired mean FP/Greedy-FP ratio to $0.985$ (median $0.983$), and the complete SR-CG algorithm beats Greedy-FP on $14/16$ instances. The paired bootstrap 95\% interval for the mean ratio is $[0.975,0.996]$. PGX accepts $10.9$ exchanges on average and performs about $4.19\times10^4$ exact candidate-swap tests. These are separate discrete FP swap evaluations rather than continuous-stage gradient evaluations. Table~\ref{tab:upgrade} summarizes the generic-tensor comparison, and Fig.~\ref{fig:pgx}(a) shows the paired stage-wise FP changes. Thus the computational benefit lies in reducing exact gradient evaluations during the continuous stage, while PGX adds a separate discrete refinement cost.

\begin{table}[t]
\centering
\caption{Generic-tensor comparison over 16 paired instances. Continuous-stage step and exact-coordinate counts exclude PGX swap tests.}
\label{tab:upgrade}
\begin{tabular}{lrrr}
\toprule
Method & FP/Greedy & steps & exact coordinates\\
\midrule
Greedy-FP & 1.000 & -- & --\\
Fixed-step CG & 1.388 & 36 & 4375\\
Rounded SR-CG & 1.141 & \textbf{16} & \textbf{2607}\\
\textbf{SR-CG} & \textbf{0.985} & \textbf{16} & \textbf{2607}\\
\bottomrule
\end{tabular}
\end{table}

\begin{figure}[t]
\centering
\includegraphics[width=0.98\linewidth]{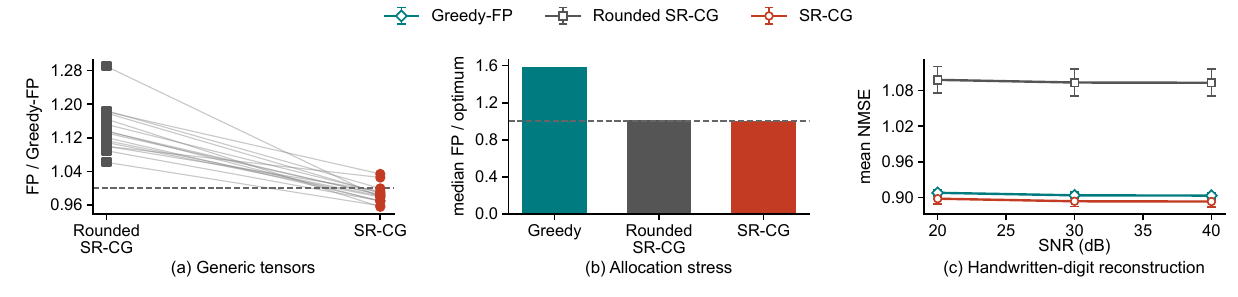}
\caption{Stage-wise and downstream comparisons. (a) Paired generic FP ratios for Rounded SR-CG and SR-CG. (b) FP relative to the enumerated optimum on the allocation-stress family. (c) Held-out handwritten-digit NMSE across signal-to-noise ratio. Error bars are standard errors over 16 trials.}
\label{fig:pgx}
\end{figure}

\subsection{Efficiency, ablation, and scaling}
\subsubsection{Residual-threshold ablation}
We vary only the residual threshold $\eta\in\{0.10,0.15,0.20,0.25,0.30\}$ while keeping $\tau_0=0.9999$ fixed. The ablation uses eight paired generic instances and reports Rounded SR-CG so that the effect of residual control is not mixed with PGX.

The corresponding mean FP/Greedy-FP values for Rounded SR-CG are $1.112,\allowbreak 1.110,\allowbreak 1.114,\allowbreak 1.119,$ and $1.131$. Over the same range, the mean number of adaptive steps decreases from $45.4$ to $13.0$ and exact-coordinate evaluations from $7545$ to $2120$. We use $\eta=0.25$ because moving from $0.20$ to $0.25$ changes FP/Greedy-FP by only $0.005$, while reducing the mean step count from $21$ to $16$ and exact-coordinate evaluations from $3447$ to $2596$. Fig.~\ref{fig:eta} summarizes this quality--computation trade-off. A fine uniform discretization repeatedly rebuilds directions that remain accurate. The exact directional polynomial identifies these regions and allows longer steps.

\begin{figure}[t]
\centering
\includegraphics[width=0.98\linewidth]{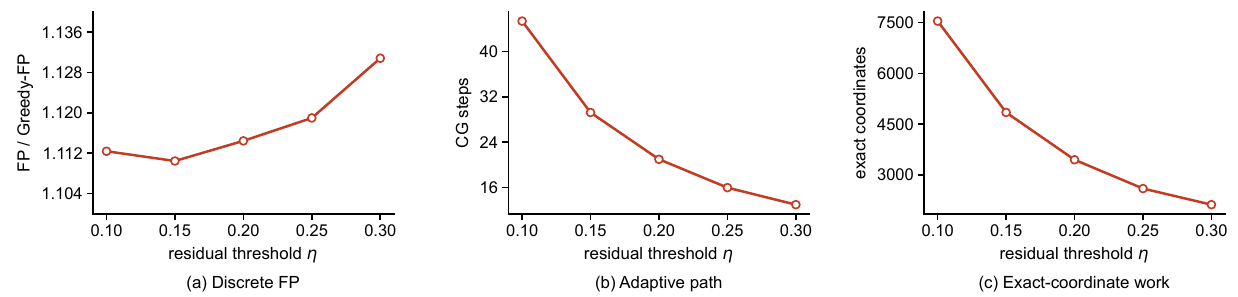}
\caption{Residual-threshold ablation for Rounded SR-CG over eight paired generic instances. Points report means.}
\label{fig:eta}
\end{figure}

\subsubsection{Scaling and normalized query savings}
For scaling, we use four paired seeds at $N_\Sigma\in\{300,600,1200,2400,3000,3600\}$ and report Rounded SR-CG. To compare an adaptive path with the 36-step fixed reference, we use
\begin{equation}
 1-\frac{\sum_tq_t}{36N_\Sigma}. \label{eq:fixednormsaving}
\end{equation}
The resulting exact-coordinate savings remain $62.6$--$72.3\%$ across the tested sizes. Table~\ref{tab:scalingnew} reports the numerical results, while Fig.~\ref{fig:diag} shows the corresponding query reduction, adaptive path length, and quality--query trade-off. The adaptive path length approaches the 36-step reference as $N_\Sigma$ increases, while the discrete FP gap is larger on some of the largest cases. The query reduction remains substantial over the tested sizes, while the FP gap varies with problem size.

\begin{table}[t]
\centering
\caption{Scaling of Rounded SR-CG over four paired seeds.}
\label{tab:scalingnew}
\begin{tabular}{rrrr}
\toprule
$N_\Sigma$ & steps & saved by \eqref{eq:fixednormsaving} & FP/Greedy\\
\midrule
300  & 17.0 & 72.3\% & 1.207\\
600  & 22.3 & 68.2\% & 1.175\\
1200 & 27.8 & 65.6\% & 1.340\\
2400 & 33.0 & 64.6\% & 1.423\\
3000 & 34.8 & 65.0\% & 1.342\\
3600 & 35.5 & 62.6\% & 1.489\\
\bottomrule
\end{tabular}
\end{table}

\begin{figure}[t]
\centering
\includegraphics[width=0.98\linewidth]{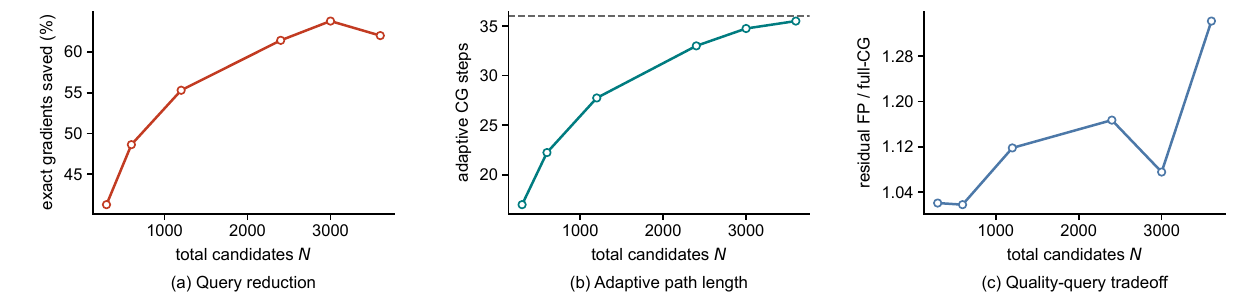}
\caption{Scaling diagnostics for Rounded SR-CG: exact-coordinate saving relative to the 36-step reference, adaptive path length, and fractional FP relative to full-gradient continuous greedy.}
\label{fig:diag}
\end{figure}

\subsection{Cross-mode allocation stress test}
The stress experiment uses a two-mode rank-two family designed to expose premature cross-mode commitment. The retained lower quotas are fixed, but the final modal cardinalities are not. Each method must decide how to distribute the remaining shared budget across the two modes. Each factor row is a two-dimensional unit vector parameterized by an angle $\theta_i^{(r)}$. For each trial,
\[
\widetilde\theta_i^{(r)}
=\theta_i^{(r)}+\varepsilon_i^{(r)},
\qquad
\varepsilon_i^{(r)}\sim\mathcal N(0,\sigma^2),
\]
independently across rows, and the perturbed row is $[\cos\widetilde\theta_i^{(r)},\sin\widetilde\theta_i^{(r)}]$. Hence $\sigma$ is the angular-perturbation standard deviation in radians and all rows remain unit norm. The retained lower quotas are $(2,2)$ and the total retained budget is $L=8$.

The targeted 20-instance stress experiment uses $\sigma=0.02$. On these 20 instances, Greedy-FP has median FP/optimum $1.588$, Rounded SR-CG has median $1.013$, and SR-CG reaches $1.000$. Rounded SR-CG recovers the optimal cross-mode allocation in $12/20$ instances, whereas SR-CG reaches the optimal allocation in all 20. The gap reflects the allocation mechanism rather than the ability to allocate across modes. Greedy-FP forms the modal cardinalities sequentially from current discrete marginals, whereas SR-CG can distribute fractional mass across modes before rounding.
PGX can then reuse the final fractional state to prioritize a compact
removal frontier and correct a small number of discrete decisions
through exact FP-decreasing exchanges.

We next vary $\sigma\in\{0.01,0.02,0.03,0.05,0.08,0.12\}$, using 24 paired instances at each level. SR-CG maintains median FP/optimum $1.000$ throughout the sweep. The Greedy-FP median decreases from $1.586$ at $\sigma=0.01$ to $1.423$ at $\sigma=0.12$. SR-CG's win rates against Greedy-FP are $100,\allowbreak 100,\allowbreak 100,\allowbreak 95.8,\allowbreak 95.8,$ and $83.3\%$, while exact allocation recovery decreases from $100\%$ to $54.2\%$. The full perturbation sweep is shown in Fig.~\ref{fig:stressneigh}, and the stage-wise comparison at the representative setting is included in Fig.~\ref{fig:pgx}(b).

As $\sigma$ increases, several allocations become nearly equivalent in FP, so exact recovery of one enumerated allocation becomes less informative. SR-CG nevertheless keeps the median FP near the optimum, showing that the advantage is not tied to a single hand-crafted allocation. The fractional path can delay a cross-mode decision and select among these near-equivalent designs after more global information has accumulated. LS MSE follows the same ordering, with a smaller separation because FP is an indirect surrogate for reconstruction loss.

\begin{figure}[t]
\centering
\includegraphics[width=0.98\linewidth]{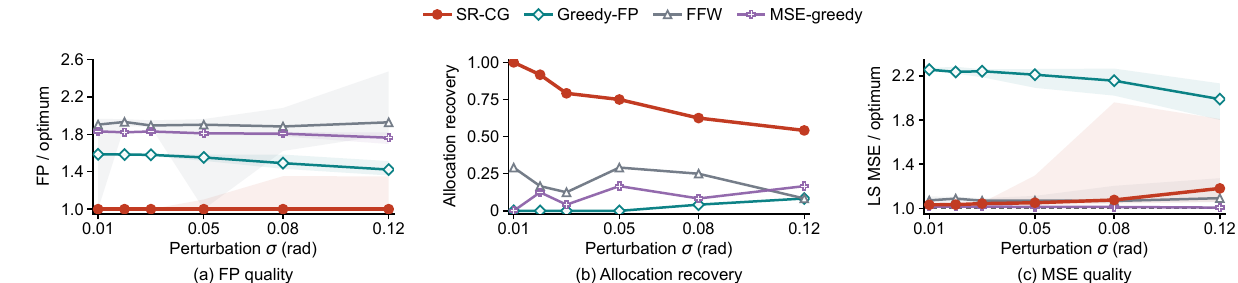}
\caption{Cross-mode allocation stress test over 24 paired instances per angular-perturbation level $\sigma$. Curves report medians and shaded regions show interquartile ranges.}
\label{fig:stressneigh}
\end{figure}

\subsection{Handwritten-digit reconstruction}
We use the 1,797-image \texttt{load\_digits} data set distributed with scikit-learn, which is a copy of the test portion of the UCI Optical Recognition of Handwritten Digits data set \cite{alpaydin1998digits}. Each sample is an $8\times8$ grayscale image with integer pixel values in $[0,16]$, which we rescale to $[0,1]$. For each of 16 trials, a random 80/20 train/test split is formed. Only the training images are used to construct the sampling design. After subtracting the training mean image, row and column covariance matrices are formed, and their leading three eigenvectors define rank-3 spatial factors. The factor rows are then normalized. We use $\kappa_1=\kappa_2=4$ and $L=9$, so every feasible design retains either $4\times5$ or $5\times4$, i.e., 20 sampled pixels.

Held-out images are reconstructed by least squares at signal-to-noise ratio (SNR) levels of $20$, $30$, and $40$ dB. Before solving the least-squares problem, the sampled modal factors and the Kronecker sensing matrix are checked for full column rank. For a test image $\mathbf X$, training mean $\bar{\mathbf X}$, and reconstruction $\widehat{\mathbf X}$, we report
\[
\operatorname{NMSE}
=
\frac{\|\widehat{\mathbf X}-\mathbf X\|_F^2}
{\|\mathbf X-\bar{\mathbf X}\|_F^2}.
\]

At 30 dB, Greedy-FP and SR-CG obtain mean NMSE values of $0.904$ and $0.894$, respectively, while Rounded SR-CG obtains $1.094$. At 20/40 dB, Greedy-FP obtains $0.908/0.903$ and SR-CG obtains $0.898/0.893$. Thus SR-CG has lower mean NMSE than Greedy-FP at all three tested SNRs. It wins $10/16$ paired trials at each SNR, and the paired bootstrap interval for the mean difference includes zero. Fig.~\ref{fig:pgx}(c) shows the mean NMSE across the three tested SNRs. The experiment shows a consistent mean improvement, but the paired evidence does not establish a statistically significant advantage. This distinction is expected because FP controls the geometry of the sensing design, whereas reconstruction also depends on the signal realization, noise, and conditioning of the sampled inverse problem.

As a control, we also apply unrestricted exact FP-decreasing one-exchange refinement to other initializations. At 30 dB, FFW changes from mean NMSE $1.421$ to $0.873$, whereas Greedy-FP changes from $0.904$ to $0.937$. Exact FP-decreasing refinement can therefore affect NMSE differently depending on the initialization.
Within SR-CG, we treat PGX as an $x_T$-guided discrete FP refinement
stage and report reconstruction separately from the formal FP
guarantee.

Overall, the experiments separate continuous-stage computational savings, discrete FP refinement, and downstream reconstruction. SR-CG reduces exact gradient evaluations during the continuous stage, while PGX uses separate discrete FP swap evaluations to improve the final design. The fractional trajectory retains information about alternative cross-mode allocations before the modal cardinalities become discrete. On handwritten-digit reconstruction, SR-CG achieves lower mean NMSE than Greedy-FP at every tested SNR, although the paired bootstrap analysis does not establish a statistically significant advantage. The reconstruction study therefore provides downstream evidence for the resulting sampling designs without implying a per-instance reconstruction guarantee.

\section{Conclusion}
We study FP-based tensor sampling under a shared budget whose modal
cardinalities are optimized jointly rather than fixed in advance.
\alg reduces unnecessary exact-gradient evaluations while retaining a
state-dependent fractional allocation until rounding. Within this
design, Shapley values serve only as query priorities, so the accepted
directions remain driven by current-state gradient information. The
finite-step guarantee approaches the classical $1-1/e$ factor as
direction certification becomes exact and finite-step residual error
vanishes. Experiments show lower continuous-stage gradient work and
better FP designs, with clearer gains on instances where the cross-mode
budget allocation is difficult to determine. On handwritten-digit
reconstruction, \alg attains lower mean NMSE than Greedy-FP at all
tested noise levels, although the paired difference is not statistically
significant.
The computational savings concern exact gradient evaluations in the
continuous stage. Within SR-CG, PGX reuses the final fractional state
to restrict the post-rounding FP exchange search and incurs a separate
discrete-swap cost.

\section*{Declarations}
\textbf{Funding.} The work was supported by the Major Scientific and Technological Innovation Platform Project of Hunan Province (2024JC1003). \textbf{Conflict of interest.} The authors declare no relevant financial or nonfinancial conflicts of interest. \textbf{Ethical compliance.} This work reports computational experiments on synthetic tensors and a publicly available machine-learning benchmark. No new human or animal subjects were recruited or studied, and no ethical approval was required for this computational study.

\bibliographystyle{IEEEbib}
\bibliography{references}
\end{document}